\documentclass[a4paper,12pt]{article}
\usepackage{amsfonts}
\usepackage{amsmath}
\usepackage{amssymb}
\usepackage{amsthm}
\usepackage{graphicx}
\usepackage{tikz}
\usepackage[inline]{enumitem}
\usepackage[hidelinks]{hyperref}
\usepackage[labelformat=simple]{subcaption}

\usetikzlibrary{arrows,shapes, positioning, matrix, patterns}

\newtheorem{prop}{Proposition}[section]
\newtheorem{lem}[prop]{Lemma}
\newtheorem{coro}[prop]{Corollary}
\newtheorem{teo}[prop]{Theorem}

\theoremstyle{definition}
\newtheorem{defix}[prop]{Definition}
\newenvironment{defi}
  {\pushQED{\qed}\defix}
  {\popQED\enddefix}
\newtheorem{exex}[prop]{Example}
\newenvironment{exe}
  {\pushQED{\qed}\exex}
  {\popQED\endexex}

\newtheorem{obsx}[prop]{Remark}
\newenvironment{obs}
  {\pushQED{\qed}\obsx}
  {\popQED\endobsx}

\newcommand{\depth}{\textnormal{depth}}
\newcommand{\lcm}{\textnormal{l.c.m.}}

\providecommand{\keywords}[1]{\textit{Keywords:  } #1}

\providecommand{\MSC}[1]{\textit{2010 Mathematics subject classification:  } #1}

\title{Characterization of Fundamental Networks}

\author{Manuela A D Aguiar
\thanks{Faculdade de Economia, Universidade do Porto, Rua Dr Roberto Frias, 4200-464 Porto, Portugal}\ \thanks{Dep. Matem\'{a}tica, Centro de Matem\'{a}tica, Universidade do Porto, Rua do Campo Alegre, 687, 4169-007 Porto, Portugal}, Ana P S Dias
\footnotemark[2]\ \thanks{Centro de Matem\'{a}tica da Universidade do Porto, Rua do Campo Alegre, 687, 4169-007 Porto, Portugal
\newline
(maguiar@fep.up.pt, apdias@fc.up.pt, ptcsoares@fc.up.pt)\newline 
The authors were partially supported by CMUP (UID/MAT/00144/2013), which is funded by FCT (Portugal) with national (MEC) and European structural funds (FEDER), under the partnership agreement PT2020.  PS acknowledges the funding by FCT through the PhD grant PD/BD/105728/2014. } and Pedro Soares\footnotemark[2]\ \footnotemark[3]
}

\date{\vspace{-13mm}}

\begin{document}

 \maketitle 

\begin{abstract}
In the framework of coupled cell systems, a coupled cell network describes graphically the dynamical dependencies between individual dynamical systems, the cells. The fundamental network of a network reveals the hidden symmetries of that network. 
Subspaces defined by equalities of coordinates which are flow-invariant for any coupled cell system consistent with a network structure are called the network synchrony subspaces.
Moreover, for every synchrony subspaces, each network admissible system restricted to that subspace is a dynamical systems consistent with a smaller network. The original network is then said to be a lift of the smaller network.
We characterize networks such that: its fundamental network is a lift of the network; the network is a subnetwork of its fundamental network, and the network is a fundamental network.
The size of cycles in a network and the distance of a cell to a cycle are two important properties concerning the description of the network architecture. In this paper, we relate these two architectural properties in a network and its fundamental network.

\vspace{1em}
\hspace{-1.8em}
\keywords{Fundamental Networks; Network Connectivity; Rings.}

\vspace{1em}
\hspace{-1.8em}
\MSC{05C25; 05C60; 05C40}
\end{abstract}

\section{Introduction}
Coupled cell networks describe influences between cells. A network is represented by a graph where each cell and each edge have a specific type. A cell type defines the nature of a cell, and an edge type defines the nature of the influence. A dynamical system that respects the network structure is a coupled cell system admissible by the network. Stewart, Golubitsky and Pivato \cite{SGP03}, and Golubitsky, Stewart and T\"{o}r\"{o}k \cite{GST05} formalized the concepts of coupled cell network and coupled cell system. They showed that there exists an intrinsic relation between coupled cell systems and coupled cells networks, proving in particular, that robust patterns of synchrony of cells are in one-to-one correspondence to balanced colorings of cells in the network -- see \cite[theorem 6.5]{SGP03}.
Coupled cell networks and coupled cell systems have been addressed, for example, from the bifurcation point of view, \cite{ADGL09,G14,K09,M14}.

Recently, Rink and Sanders \cite{RS14,RS15} and Nijholt, Rink and Sanders \cite{NRS14} developed some dynamical techniques for homogenous networks with asymmetric inputs, i.e., networks where all cells have the same type and each cell receives only one edge of each type. 
When the network has a semi-group structure, they have calculated normal forms of coupled cell systems and used the hidden symmetries of the network to derive Lyapunov-Schmidt reduction that preserves hidden symmetries.
They have introduced the concept of fundamental network which reveals the hidden symmetries of a network (see definition~\ref{deffundnet} of \S~\ref{fundsec}).
 A fundamental network is a Cayley Graph of a semi-group. The dynamics associated to a fundamental network can be studied using the revealed hidden symmetries. 
Moreover, the dynamics associated to a network can be derived from the dynamics associated to its fundamental network, \cite[theorem 10.1]{RS14}.

The one-to-one correspondence between balanced colorings and synchrony subspaces leads to the definition of quotient network, such that every dynamics associated to the quotient network is the restriction to a synchrony subspace of the dynamics associated to the original network. A subnetwork of a given network is a network whose set of cells is a subset of the cells of the given network and the respective incoming edges, such that the cells are not influenced by any cell outside the subnetwork. Thus, the dynamics associated to the cells in a subnetwork is independent of the dynamics associated to the other cells. DeVille and Lerman \cite{DL15} highlighted the concepts of quotient network and subnetwork using network fibrations, i.e., functions between networks that respect their structure. In particular, they showed that every surjective network fibration defines a quotient network and every injective network fibration defines a subnetwork (\S~\ref{frisec}).

In this work, we will focus on the relation between a homogenous network and its fundamental network. The work is divided in two independent parts.
 In the first part, we show that the fundamental network construction preserves the quotient network relation and transforms the subnetwork relation in the quotient network relation (\S~\ref{fundfrisec}). Moreover, we characterize the networks such that: its fundamental network is a lift of the network (\S~\ref{fundliftsec}); the network is a subnetwork of its fundamental network (\S~\ref{fundsubsec}); and the network is a fundamental network (\S~\ref{fundchasec}). 
In order to do that, we introduce the properties of backward connectivity and transitivity for a cell. The backward connectedness for a cell means that we can reach that cell from any other cell in the network. 
This signifies that the dynamics associated to that cell is, directly or indirectly, affected by the dynamics associated to every other cell in the network.
 The transitivity for a cell is the existence of network fibrations pointing that cell to any other cell. This property is similar to the vertex-transitivity used in the characterization of Cayley-Graphs of groups \cite[\S 16]{B74}. The vertex-transitivity is the ability of interchanging any two nodes using a bijective fibration, which reveals the symmetries of a graph.

In the second part, we relate the architecture of a network and of its fundamental network. In particular, we study two concepts of a network's architecture: cycles in the network and the distance of cells to a cycle (\S~\ref{arcsec}). We denote by rings the cycles in the network involving only one edge type, and by depth the maximal distance of any cell to a ring. Ring networks have been studied, for example, in Ganbat \cite{G14} and Moreira \cite{M14}. We start by looking to networks having a group structure (\S~\ref{strufundsec}). Then we show that a network and its fundamental network have equal depth (\S~\ref{depthfundsec}), and that the size of the rings in a fundamental network is a (least common) multiple of the size of some network rings (\S~\ref{ringsfundsec}). Last, we describe the architecture of the fundamental networks of networks that have only one edge type.

The text is organized as follows. Sections~\ref{ccn}, \ref{fundsec} and \ref{frisec} review the concepts of coupled cell networks, fundamental networks and network fibrations, respectively. Section~\ref{fundfrisec} characterizes fundamental networks. Section~\ref{arcsec} defines rings and depth of  a network. Finally, \S~\ref{strufundsec} relates rings and depth of a network and of its fundamental network.

\section{Coupled cell networks}\label{ccn}

In this section, we recall a few facts concerning coupled cell networks following \cite{GST05,SGP03}. We also introduce the notion of backward connected network.

A \emph{directed graph} is a tuple $G=(C,E,s,t)$, where $c\in C$ is a cell and $e \in E$ is a directed edge from the source cell, $s(e)$, to the target cell, $t(e)$. We assume that the sets of cells and edges are finite.
The \emph{input set} of a cell $c$, denoted by $I(c)$, is the set of edges that target $c$.
Following \cite[definition 2.1.]{GST05} and imposing that cells of the same type are input equivalent we define (coupled cell) network.

\begin{defi}
A \emph{(coupled cell) network} $N=(G,\sim_C,\sim_E)$ is a directed graph, $G$, together with two equivalence relations: one on the set  of cells, $\sim_C$, and another on the set of edges, $\sim_E$. The \emph{cell type} of a cell is its $\sim_C$-equivalence class and the \emph{edge type} of an edge is its $\sim_E$-equivalence class. It is assumed that:\\
(i) edges of the same type have source cells of the same type and target cells of the same type;\\
(ii)  cells of the same type are input equivalent. That is, if two cells have the same cell type, then there is an edge type preserving bijection between their input sets. 
\end{defi}

We say that a network is a \emph{homogeneous network} whenever there is only one cell type. A network is a \emph{homogeneous network with asymmetric inputs} if each cell receives exactly one edge of each edge type. We will focus our interest in homogeneous networks with asymmetric inputs.

In \cite{RS14}, Rink and Sanders pointed out that a homogeneous network with asymmetric inputs can be represented by functions $\sigma_{i}: C \rightarrow C$, for each edge type $i$, such that there is an edge with type $i$ from $\sigma_{i}(c)$ to $c$.
We write $\sigma=[a_1\;\dots\;a_n]$ for the function $\sigma:\{1,\dots,n\}\rightarrow \{1,\dots,m\}$ such that $\sigma(j)=a_j$, for $j=1,\dots,n$. For examples of homogeneous networks with asymmetric inputs see figure~\ref{fig::hnai}, where distinct edge types are represented by different symbols.

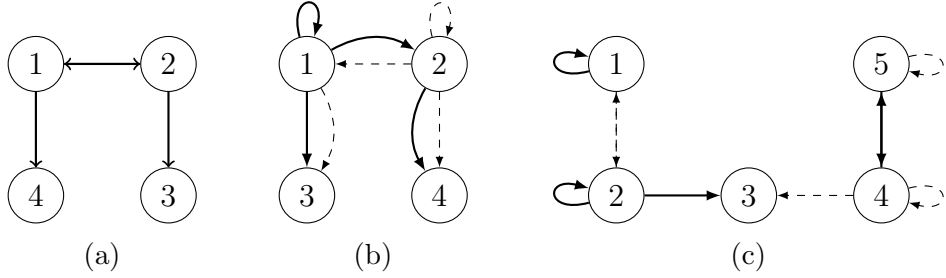
\begin{figure}[h]
\begin{subfigure}[t]{0.25\textwidth}
\centering
\begin{tikzpicture}
\node (n1) [draw, circle]               {1};
\node (n2) [draw, circle] [right=of n1]  {2};
\node (n3) [draw, circle] [below=of n2] {3}; 
\node (n4) [draw, circle] [left=of n3] {4}; 
\draw[->, thick] (n2) -- (n1);
\draw[->, thick] (n1) -- (n2);
\draw[->, thick] (n2) -- (n3);
\draw[->, thick] (n1) -- (n4);
\end{tikzpicture}
\caption{}
\label{fig:hnai1}
\end{subfigure}
\begin{subfigure}[t]{0.25\textwidth}
\centering
\begin{tikzpicture}
\node (n1) [shape=circle,draw]  {1};
\node (n2) [circle,draw]  [right=of n1] {2};
\node (n3) [shape=circle,draw] [below=of n1] {3};
\node (n4) [circle,draw]  [right=of n3] {4};
\draw[-latex, thick] (n1) to [loop above] (n1);
\draw[-latex, thick] (n1) to [bend left] (n2);
\draw[-latex, thick] (n1) to  (n3);
\draw[-latex, thick] (n2) to [bend right] (n4);
\draw[-latex, dashed] (n2) to  (n1);
\draw[-latex, dashed] (n2) to [loop above] (n2);
\draw[-latex, dashed] (n1) to [bend left] (n3);
\draw[-latex, dashed] (n2) to  (n4);
\end{tikzpicture}
\caption{}
\label{fig:hnai2}
\end{subfigure}
\begin{subfigure}[t]{0.45\textwidth}
\centering
\begin{tikzpicture}
\node (n1) [shape=circle,draw]  {1};
\node (n2) [circle,draw]  [below=of n1] {2};
\node (n3) [shape=circle,draw] [right=of n2] {3};
\node (n4) [circle,draw]  [right=of n3] {4};
\node (n5) [circle,draw]  [above=of n4] {5};

\draw[-latex, thick] (n1) to [loop left] (n1);
\draw[-latex, thick] (n2) to [loop left] (n2);
\draw[-latex, thick] (n2) to  (n3);
\draw[-latex, thick] (n5) to  (n4);
\draw[-latex, thick] (n4) to  (n5);

\draw[-latex, dashed] (n2) to  (n1);
\draw[-latex, dashed] (n1) to  (n2);
\draw[-latex, dashed] (n4) to  (n3);
\draw[-latex, dashed] (n4) to [loop right] (n4);
\draw[-latex, dashed] (n5) to [loop right] (n5);
\end{tikzpicture}
\caption{}
\label{fig:hnai3}
\end{subfigure}
\caption{Homogeneous networks with asymmetric inputs: \subref{fig:hnai1} network with one edge type represented by the function $\sigma_1=[2\;1\;2\;1]$; \subref{fig:hnai2} network with two edge types, where the solid edges are represented by $\sigma_1=[1\;1\;1\;2]$ and the dashed edges are represented by $\sigma_2=[2\;2\;1\;2]$; \subref{fig:hnai3} network represented by the functions $\sigma_1=[1\;2\;2\;5\;4]$, for solid edges, and $\sigma_2=[2\;1\;4\;4\;5]$, for dashed edges. The network \subref{fig:hnai3} is backward connected, and the networks \subref{fig:hnai1} and \subref{fig:hnai2} are not.}
\label{fig::hnai}
\end{figure}

A \emph{directed path} in a network $N$ is a sequence $(c_0,c_1,\dots,c_{m-1},c_m)$ of cells in $N$ such that for every $j=1,\dots,m$ there is an edge in $N$ from $c_{j-1}$ to $c_j$. 

\begin{obs}
Compositions of representative functions define directed paths in the network. Let $N$ be a homogenous network with asymmetric inputs represented by the functions  $\left(\sigma_i\right)_{i=1}^k$. There exists a directed path from cell $c$ to cell $d$ if and only if there are $1\leq j_1,\dots,\ j_m\leq k$ such that
\[\sigma_{j_m}\circ\dots \circ \sigma_{j_1}(d)=c.\qedhere\]
\end{obs}
%


\begin{defi}
We say that a network $N$ is \emph{backward connected for a cell} $c$ if for any cell $c'\neq c$ there exists a directed path between $c'$ and $c$.
The network $N$ is \emph{backward connected} if it is backward connected for some cell.
\end{defi}


\begin{exe}
Consider the networks in figure~\ref{fig::hnai}. For the network in figure~\ref{fig:hnai1}, there is no directed path from cell $4$ to cells $1$, $2$ and $3$, neither from cell $3$ to cells $1$, $2$ and $4$. Thus the network is not backward connected. Similarly, we see that the network in figure~\ref{fig:hnai2} is not backward connected. Now, for the network in figure~\ref{fig:hnai3}, there is a directed path starting in cell $1$, $2$, $4$ or $5$ to cell $3$. Thus, the network is backward connected for cell $3$.
\end{exe}


Following \cite{NRS14}, the input network for a cell of a network contains the cells that affect, directly or indirectly, that cell.
The \emph{input network} for $c\in C$, denoted by $N_{(c)}$, is the network with set of cells $C_{(c)}$ and set of edges $E_{(c)}$, where 
$$C_{(c)}=\left\{ c\right\}\cup\left\{c'\in C\  |\  \textrm{exists a  directed path in }N\textrm{ from }c' \textrm{ to }c \right\},$$$$ E_{(c)}=\left\{e\in E\  |\  t(e)\in C_{(c)}\right\}.$$

Observe that every input network for a cell is backward connected for that cell. See figure~\ref{fig:subhnai3} for an example.

\begin{figure}[h]
\centering
\begin{tikzpicture}
\node (n1) [draw, shape=circle] {1};
\node (n2) [draw, shape=circle] [right=of n1] {2};

\draw[-latex, dashed] (n2) to  (n1);
\draw[-latex, dashed] (n1) to (n2);

\draw[-latex, thick] (n1) to [loop left] (n1);
\draw[-latex, thick] (n2) to [loop right] (n2);
\end{tikzpicture}
\caption{Input network of the network in figure~\ref{fig:hnai3} for cell $1$ (and for cell $2$). It is backward connected for cell $1$ (and for cell $2$).}
\label{fig:subhnai3}
\end{figure}
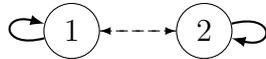

\section{Fundamental networks}\label{fundsec}

In this section, we recall the definition of fundamental network of a homogenous network with asymmetric inputs introduced by Nijholt et al. \cite{NRS14}. We present some examples of fundamental networks and remark that every fundamental network is backward connected. 

 The identity function in $C$ is denoted by $Id_C$, and we omit the subscript when it is clear from the context.

\begin{defi}[{\cite[definition  6.2]{NRS14}}]\label{deffundnet}
 Let $N$ be a homogeneous network with asymmetric inputs represented by the functions $\left(\sigma_i:C\rightarrow C\right)_{i=1}^k$. The \emph{fundamental network} of $N$ is the network  $\tilde{N}$ where the set of cells, $\tilde{C}$, is the semi-group generated by $Id$ and $\left(\sigma_i\right)_{i=1}^k$, and $\tilde{N}$ is represented by the functions $$\left(\tilde{\sigma_i}:\tilde{C}\rightarrow \tilde{C}\right)_{i=1}^k,$$ defined by $\tilde{\sigma_i}(\tilde{c})=\sigma_i\circ \tilde{c}$, for $\tilde{c}\in\tilde{C}$ and $i=1,\dots,k$.
\end{defi}


\begin{exe}\label{exefund}
Consider the network in figure~\ref{fig:hnai1}. This network is represented by the function $\sigma_1=[2\;1\;2\;1]$. Note that $\sigma_1^3=\sigma_1$, and the semi-group generated by $\sigma_1$ and $Id$ is $$\tilde{C}=\left\{Id, \sigma_1, \sigma_1^2\right\}.$$
The representative function, $\tilde{\sigma_1}$, of the fundamental network is obtained from the composition of $\sigma_1$ with each element of $\tilde{C}$:  $\tilde{\sigma}(\sigma_1^2)=\sigma_1$ and $\tilde{\sigma}(\sigma_1^j)=\sigma^{j+1}$, when $j=0,1$. The fundamental network is represented graphically in figure~\ref{fig:fundhnai1}.
\end{exe}

Figure~\ref{fig::fundhnai} displays the fundamental networks of the networks in figures~\ref{fig::hnai} and~\ref{fig:subhnai3}. Note that all the fundamental network in figure~\ref{fig::fundhnai} are backward connected for $Id$.

\begin{figure}[h]
\begin{subfigure}[t]{0.49\textwidth}
\centering
\begin{tikzpicture}
\node (n1) [draw, shape=circle,label=center:$Id$] {\phantom{0}};
\node (n2) [draw, shape=circle,label=center:$\sigma_1$] [above=of n1] {\phantom{0}};
\node (n3) [draw, shape=circle,label=center:$\sigma_1^2$] [right=of n2] {\phantom{0}};

\draw[-latex, thick] (n2) to (n1);
\draw[-latex, thick] (n3) to (n2);
\draw[-latex, thick] (n2) to (n3);
\end{tikzpicture}
\caption{}
\label{fig:fundhnai1}
\end{subfigure}
\begin{subfigure}[t]{0.49\textwidth}
\centering
\begin{tikzpicture}
\node (n4) [circle,draw,label=center:$\sigma_1^2$]  {\phantom{0}};
\node (n3) [shape=circle,draw,label=center:$\sigma_2$] [below=of n4] {\phantom{0}};
\node (n1) [shape=circle,draw,label=center:$Id$] [right=of n3] {\phantom{0}};
\node (n2) [circle,draw,label=center:$\sigma_1$]  [right=of n1] {\phantom{0}};
\node (n5) [circle,draw,label=center:$\sigma_2^2$]  [above=of n2] {\phantom{0}};

\draw[-latex, thick] (n2) to  (n1);
\draw[-latex, thick] (n4) to (n2);
\draw[-latex, thick] (n4) to  (n3);
\draw[-latex, thick] (n4) to [loop above] (n4);
\draw[-latex, thick] (n4) to [bend left] (n5);
\draw[-latex, dashed] (n3) to (n1);
\draw[-latex, dashed] (n5) to (n2);
\draw[-latex, dashed] (n5) to (n3);
\draw[-latex, dashed] (n5) to  (n4);
\draw[-latex, dashed] (n5) to [loop above] (n5);
\end{tikzpicture}
\caption{}
\label{fig:fundhnai2}
\end{subfigure}
\begin{subfigure}[t]{0.84\textwidth}
\centering
\begin{tikzpicture}
\node (n1) [draw, shape=circle] {6};
\node (n2) [draw, shape=circle] [right=of n1] {1};
\node (n3) [draw, shape=circle] [below=of n1] {7};
\node (n4) [draw, shape=circle] [right=of n3] {2};
\node (n5) [draw, shape=circle] [right=of n4] {3};
\node (n6) [draw, shape=circle] [right=of n5] {4};
\node (n7) [draw, shape=circle] [right=of n6] {8};
\node (n8) [draw, shape=circle] [above=of n7] {9};
\node (n9) [draw, shape=circle] [left=of n8] {5};

\draw[-latex, thick] (n2) to  (n1);
\draw[-latex, thick] (n1) to  (n2);
\draw[-latex, thick] (n4) to  (n3);
\draw[-latex, thick] (n3) to  (n4);
\draw[-latex, thick] (n4) to  (n5);
\draw[-latex, thick] (n9) to  (n6);
\draw[-latex, thick] (n8) to  (n7);
\draw[-latex, thick] (n7) to  (n8);
\draw[-latex, thick] (n6) to  (n9);

\draw[-latex, dashed] (n3) to  (n1);
\draw[-latex, dashed] (n4) to  (n2);
\draw[-latex, dashed] (n1) to  (n3);
\draw[-latex, dashed] (n2) to  (n4);
\draw[-latex, dashed] (n6) to  (n5);
\draw[-latex, dashed] (n7) to  (n6);
\draw[-latex, dashed] (n6) to  (n7);
\draw[-latex, dashed] (n9) to  (n8);
\draw[-latex, dashed] (n8) to  (n9);
\end{tikzpicture}
\caption{}
\label{fig:fundhnai3}
\end{subfigure}
\begin{subfigure}[t]{0.15\textwidth}
\centering
\begin{tikzpicture}
\node (n1) [draw, shape=circle,label=center:$Id$] {\phantom{0}};
\node (n2) [draw, shape=circle,label=center:$\sigma_2$] [below=of n1] {\phantom{0}};

\draw[-latex, dashed] (n2) to  (n1);
\draw[-latex, dashed] (n1) to (n2);

\draw[-latex, thick] (n1) to [loop left] (n1);
\draw[-latex, thick] (n2) to [loop left] (n2);
\end{tikzpicture}
\caption{}
\label{fig:fundsubhnai3}
\end{subfigure}
\caption{Fundamental networks of the networks in figure~\ref{fig:hnai1}, \subref{fig:hnai2}, \subref{fig:hnai3} and figure~\ref{fig:subhnai3}, respectively. The cells $1,\dots,9$ in \subref{fig:fundhnai3} correspond to the functions $\sigma_2\circ\sigma_1$, $\sigma_1$, $Id$, $\sigma_2$, $\sigma_1\circ \sigma_2$, $\sigma_2\circ\sigma_1^2$, $\sigma_1^2$, $\sigma_2^2$, $\sigma_1\circ\sigma_1^2$, respectively. In \S~\ref{frisec}, we see that the fundamental network in: \subref{fig:fundhnai1} is a quotient network and a subnetwork of the network in figure~\ref{fig:hnai1}; \subref{fig:fundhnai2} is neither a lift nor a quotient network of the network in figure~\ref{fig:hnai2}; \subref{fig:fundhnai3} is a lift of the network in figure~\ref{fig:hnai3}; \subref{fig:fundsubhnai3} is equal to the network in figure~\ref{fig:subhnai3}.}
\label{fig::fundhnai}
\end{figure}
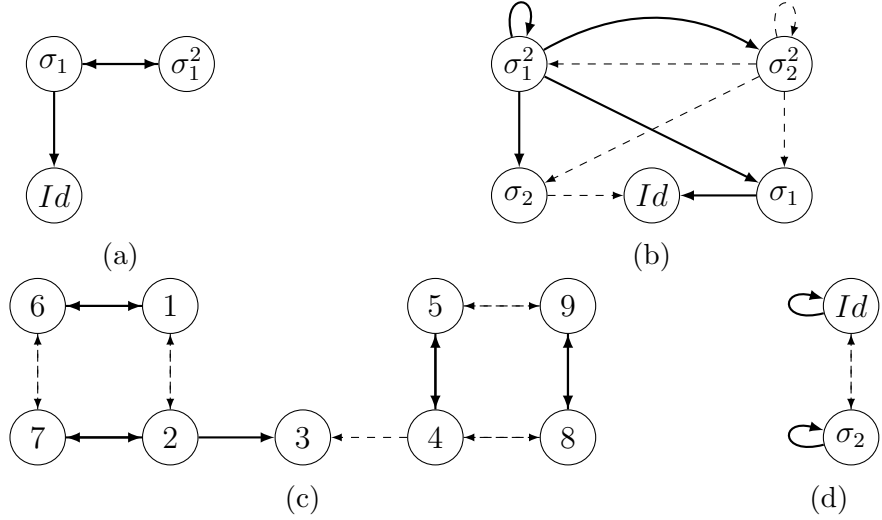

%
%
%
%
%

\begin{prop}\label{backconfund}
Every fundamental network of a homogenous network with asymmetric inputs is backward connected for $Id$.
\end{prop}

\begin{proof}
Let $N$ be a homogenous network with asymmetric inputs represented by $\left(\sigma_i\right)_{i=1}^{k}$ and $\tilde{N}$ its fundamental network. 
If $\tilde{c}\in \tilde{C}$, then $\tilde{c}=\sigma_{l_1}\circ\dots\circ\sigma_{l_m}$, where $1\leq l_i\leq k$, and 
$$\widetilde{\sigma_{l_1}}\circ\dots\circ\widetilde{\sigma_{l_m}}(Id)=\sigma_{l_1}\circ\dots\circ\sigma_{l_m}\circ Id=\tilde{c}.$$
Hence $\tilde{N}$ is backward connected for $Id$.
\end{proof}

\section{Network fibrations}\label{frisec}

In this section, we recall the definition and some properties of network fibrations. We introduce a notion of transitivity and we recall the definitions of quotient network and subnetwork. Moreover, we highlight the relations of quotient network and subnetwork with surjective and injective network fibrations, respectively.

Roughly speaking, a graph fibration is a function between graphs that preserves the orientation of the edges and the number of input edges. Precisely, let $G=(C,E,s,t)$ and $G'=(C',E',s',t')$ be two graphs. A function $\varphi:G\rightarrow G'$ is a \emph{graph fibration} if $\varphi(s(e))=s'(\varphi(e))$, $\varphi(t(e))=t'(\varphi(e))$ and $\left.\varphi\right|_{I(c)}:I(c)\rightarrow I(\varphi(c))$ is a bijection, for every $c\in C$ and $e\in E$. 

A network fibration between networks is then defined as a graph fibration preserving the cell types and the edge types:

\begin{defi}[{\cite[definition 4.1.1]{DL15}}]
Consider two networks $N=(G,\sim_C,\sim_E)$ and $N'=(G',\sim_{C},\sim_{E})$. A \emph{network fibration} $\varphi:N\rightarrow N'$ is a graph fibration between $G$ and $G'$ such that
$c\sim_C  \varphi(c)$ and $e\sim_E \varphi(e)$.

We say that $N$ and $N'$ are \emph{isomorphic}, if there is a bijective network fibration between $N$ and $N'$.
\end{defi}

We do not distinguish isomorphic networks and we will say that two networks are the same if they are isomorphic.

\begin{exe}\label{exe:netfri}
Let $N$ be the network in figure~\ref{fig:hnai1}. Denote an edge of $N$ with source $s$ and target $t$ by $(s,t)$. Consider the function  $\varphi: N\rightarrow N$ such that $\varphi(1)=1$, $\varphi(2)=\varphi(4)=2$ and $\varphi(3)=3$, and $\varphi((1,2))=\varphi((1,4))=(1,2)$, $\varphi((2,1))=(2,1)$ and $\varphi((2,3))=(2,3)$.
The function $\varphi$ is a network fibration.
\end{exe}

In the case of homogeneous networks with asymmetric inputs, the network fibrations are characterized by the following property.

\begin{prop}[{\cite[proposition 5.3]{NRS14}}]\label{NRS1453}
Let $N$ and $N'$ be homogeneous networks with asymmetric inputs with set of cells $C$ and $C'$\!, and represented by the functions  $\left(\sigma_i\right)_{i=1}^k$ and $\left(\sigma'_i\right)_{i=1}^k$, respectively.  The function $\varphi:N\rightarrow N'$ is a network fibration if and only if $$\left.\varphi\right|_C \circ \sigma_i=\sigma'_i\circ \left.\varphi\right|_C,\quad \quad i=1,\dots,k.$$
\end{prop}

\begin{exe}
Recall the network $N$ in figure~\ref{fig:hnai1} represented by the function $\sigma_1=[2\;1\;2\;1]$. Consider the network fibration, given in example~\ref{exe:netfri}, $\varphi: N\rightarrow N$ such that $\varphi=[1\;2\;3\;2]$. Observe that $\varphi \circ \sigma_1= [2\;1\;2\;1]= \sigma_1\circ \varphi$.
\end{exe}

A network fibration from a network which is backward connected for a cell $c$ is uniquely determined by the evaluation of the network fibration at cell $c$.

\begin{prop}\label{netfriback}
Let $A$ be a homogeneous network with asymmetric inputs and $\phi:A\rightarrow B$ a network fibration. If $A$ is backward connected for $c$, then the network fibration is uniquely determined by $\phi(c)$.
\end{prop}

\begin{proof}
Let $A$ be a homogeneous network with asymmetric inputs and $\phi:A\rightarrow B$ a network fibration. Then $B$ is a homogeneous network with asymmetric inputs and has the same edge types of $A$. Suppose that $A$ and $B$ are represented by the functions $\left(\sigma^1_i\right)_{i=1}^{k}$  and $\left(\sigma^2_i\right)_{i=1}^{k}$, respectively, and $A$ is backward connected for $c$. Then for every cell $d\neq c$ in $A$ there are $\sigma^1_{i_1},\dots,\sigma^1_{i_m}$ with $1\leq i_1,\dots,i_m\leq k$ such that $d=\sigma^1_{i_1}\circ\dots\circ\sigma^1_{i_m}(c)$.
By proposition~\ref{NRS1453}, we know that $\phi\circ\sigma_i^1=\sigma_i^2\circ\phi$, for $1\leq i \leq k$. Then, for every cell $d\neq c$ in $A$,

\[\phi(d)=\phi\circ \sigma^1_{i_1}\circ\dots\circ\sigma^1_{i_m}(c)=\sigma^2_{i_1}\circ\dots\circ\sigma^2_{i_m}\circ\phi(c).\qedhere\] 
\end{proof}

In the context of graphs, vertex-transitivity is the ability of interchanging two cells of a graph using a bijective graph fibration. The vertex-transitivity reveals symmetries in a graph and it was usefully used in the characterization of Cayley graphs of groups, see \cite[{\S 16}]{B74}. Here, we introduce a weaker version of transitivity that will play a similar role in the characterization of fundamental networks.

\begin{defi}
Let $N$ be a homogenous network with asymmetric inputs and $c$ a cell in $N$. We say that $N$ is \emph{transitive for $c$} if for every cell $d$ in $N$, there is a network fibration $\phi_d:N\rightarrow N$ such that $\phi_d(c)=d$. We call the network $N$ \emph{transitive}, if it is transitive for some cell.
\end{defi}

\begin{exe}Consider the networks in figure~\ref{fig::hnai}. For the network in figure~\ref{fig:hnai1}, we have the following four network fibrations from the network to itself: $\phi_1=[1\;2\;1\;2]$, $\phi_2=[2\;1\;2\;1]$, $\phi_3=[1\;2\;3\;4]$, and $\phi_4=[2\;1\;4\;3]$. Then the network is transitive for cell $3$ (and for cell $4$). For the network in figure~\ref{fig:hnai2}, there is only one network fibration from the network to itself, the identity network fibration. Thus the network is not transitive.
\end{exe}

\subsection{Surjective network fibrations}

We recall now the definition of quotient networks using balanced colorings \cite{GST05,SGP03} and establish then their relation with surjective network fibrations, \cite{BV02,DL15}.

A \emph{coloring} on the set of cells of a network defines an equivalence relation on those cells. Following \cite{GST05,SGP03}, a coloring is \emph{balanced} if for any two cells with the same color there is an edge type preserving bijection between the corresponding input sets which also preserves the color of the source cells.

Each balanced coloring defines a quotient network, see \cite[\S 5]{GST05}. The \emph{quotient network} of a network with respect to a given balanced coloring $\bowtie$, is the network where the set of equivalence classes of the coloring, $[c]_{\bowtie}$, is the set of cells and there is an edge of type $i$ from $[c]_{\bowtie}$ to $[c']_{\bowtie}$, for each edge of type $i$ from a cell in the class $[c]_{\bowtie}$ to $c'$. 
We say that a network $L$ is a \emph{lift} of $N$, if $N$ is a quotient network of $L$.

\begin{exe}
 Let $N$ be the network in figure~\ref{fig:hnai1} and $\tilde{N}$ its fundamental network displayed in figure~\ref{fig:fundhnai1}. The coloring on the set of cells of $N$ with classes $\{1,3\}$, $\{2\}$ and $\{4\}$ is balanced because  cells $1$  and $3$ receive, each, an edge from cell $2$. The quotient network of $N$ with respect to this balanced coloring is $\tilde{N}$. Hence the fundamental network is a quotient network.
\end{exe}
\begin{exe}
 The network in figure~\ref{fig:hnai3} is a quotient network of its fundamental network displayed in figure~\ref{fig:fundhnai3} with respect to the balanced coloring with classes $\{1,6\}$, $\{2,7\}$, $\{4,8\}$ and $\{5,9\}$. In this case, the fundamental network is a lift.
\end{exe}

The balanced colorings are uniquely determined by surjective network fibrations, see \cite[{theorem 2}]{BV02}, \cite[{remark 4.3.3}]{DL15} or \cite[{theorem 8.3}]{SGP03}.

\begin{prop}[{\cite[{theorem 2}]{BV02}}]\label{quosurjnetfri}
A network $Q$ is a quotient network of a network $N$ if and only if there is a surjective network fibration from $N$ to $Q$. 
\end{prop}

For completeness, we sketch the proof here. 
If $Q$ is a quotient network of a network $N$, consider the associated balanced coloring. The function from $N$ to $Q$ that project each cell into its equivalence class is a surjective network fibration. 
On the other hand, given a surjective network fibration from $N$ to $Q$, consider the coloring such that two cells have the same color, when their evaluation by the network fibration is equal. This coloring is balanced, and the quotient network of $N$ with respect to this coloring is equal to $Q$.

\begin{exe}
 Let $N$ be the network in figure~\ref{fig:hnai3} and $\tilde{N}$ its fundamental network displayed in figure~\ref{fig:fundhnai3}. The network fibration from $\tilde{N}$ to $N$ given by $\varphi=[1\;2\;3\;4\;5\;1\;2\;4\;5]$ is surjective and $N$ is a quotient network of $\tilde{N}$.
\end{exe}
\begin{exe}
 There is no surjective fibration from the network in figure~\ref{fig:hnai2} to its fundamental network displayed in figure~\ref{fig:fundhnai2} neither a surjective fibration from the fundamental network to the network. Hence, in this case, the fundamental network is neither a lift nor a quotient network of the network.
\end{exe}

\subsection{Injective network fibrations}

We consider now subnetworks and their relation with injective network fibrations. We follow \cite[\S 5.2]{DL15}.

\begin{defi}
Let $N$ and $S$ be two networks with sets of cells and edges, respectively, $C$ and $E$, and $C'$ and $E'$. Then $S$ is a \emph{subnetwork} of $N$, if $C'\subseteq C$, $E'\subseteq E$ and for every $c'\in C'$ and every edge $e\in E$ with the target cell $t(e)=c'$, we have that $e\in E'$ and the source cell $s(e)\in C'$.
\end{defi}

\begin{exe}
 Consider the network in figure~\ref{fig:hnai1} and its fundamental network displayed in figure~\ref{fig:fundhnai1}. The fundamental network is a subnetwork.
\end{exe}

\begin{obs}
Let $N$ be a network with set of cells $C$.\\
(i) For every cell $c\in C$, the input network $N_{(c)}$ is a subnetwork of $N$.\\
(ii) The union of subnetworks of $N$ is a subnetwork of $N$.
\end{obs}

\begin{exe}
 Let $N$ be the network in figure~\ref{fig:hnai3}. The restriction of $N$ to the set of cells $\{1,2,4,5\}$ is a subnetwork of $N$. That restriction corresponds to the union of the input networks for the cells $1$, $2$, $4$ and $5$.
\end{exe}

\begin{prop}[{\cite[\S 5.2]{DL15}}]\label{subinjnetfri}
A network $N'$ is a subnetwork of $N$ if and only if there is an injective network fibration from $N'$ to $N$.
\end{prop}

For completeness, we sketch the proof here.
If $N'$ is a subnetwork of $N$, then the embedding of $N'$ in $N$ is an injective network fibration. If $\varphi:N'\rightarrow N$ is an injective network fibration, then $N'$ is equal to $\varphi(N')$ which is a subnetwork of $N$.

\section{Fundamental networks and network fibrations}\label{fundfrisec}

In this section, we recall some results presented by Nijholt et al. in \cite{NRS14}. We show then that the fundamental network construction preserves the quotient network relation. Moreover, we see that the fundamental network construction does not preserve the subnetwork relation, but it transforms the subnetwork relation in the quotient network relation.

\begin{teo}[{\cite[theorem 6.4 \& remark 6.8 \& lemma 7.1]{NRS14}}]\label{NRS1464}
Let $N$ be a homogeneous network with asymmetric inputs and $\tilde{N}$ its fundamental network with set of cells $C$ and $\tilde{C}$, respectively. For every $c\in C$, there is a network fibration, $\varphi_{c}: \tilde{N}\rightarrow N$ given by 
$$\varphi_{c}(\tilde{c})=\tilde{c}(c), \quad \tilde{c}\in\tilde{C}.$$
The image of $\varphi_{c}$ is the input network $N_{(c)}$. 
Every network fibration from $\tilde{N}$ to $N$ is equal to $\varphi_{c}$ for some $c\in C$. The network $\tilde{N}$ and its fundamental $\tilde{\tilde{N}}$ are equal.
\end{teo}

We prove next that the fundamental network construction preserves the quotient network relation.

\begin{prop}\label{fundprelift}
Let $N$ be a homogeneous network with asymmetric inputs.
If $Q$ is a quotient network of $N$, then $\tilde{Q}$ is a quotient network of $\tilde{N}$.
\end{prop}

\begin{proof}
Let $N$ be a homogeneous network with asymmetric inputs and $Q$ a quotient network of $N$. By proposition~\ref{quosurjnetfri}, there exists a surjective network fibration $\phi:N\rightarrow Q$ and $Q$ is a homogeneous network with asymmetric inputs. Suppose that $N$ and $Q$ are represented by $\left(\sigma_l\right)_{l=1}^{k}$ and $\left(\gamma_l\right)_{l=1}^{k}$, respectively. By proposition~\ref{NRS1453},
$$\phi\circ \sigma_i=\gamma_i\circ \phi,\quad \quad i=1,\dots,k.$$

Define the function $\tilde{\phi}:\tilde{N}\rightarrow \tilde{Q}$ such that  $\tilde{\phi}(Id_N)=Id_Q$ and for every cell $\sigma$ in $\tilde{N}$ such that $\sigma=\sigma_{i_1}\circ\dots\circ\sigma_{i_m}$ for some $1\leq i_1,\dots,i_m\leq k$, then $\tilde{\phi}$ is given by $\tilde{\phi}(\sigma)=\gamma_{i_1}\circ\dots\circ\gamma_{i_m}$. As we show next, $\tilde{\phi}$ is well-defined,  surjective and a network fibration.

Suppose that $\sigma=\sigma_{i_1}\circ\dots\circ\sigma_{i_m}=\sigma_{j_1}\circ\dots\circ\sigma_{j_{m'}}$, where $1\leq i_1,\dots,i_m\leq k$ and $1\leq j_1,\dots,j_{m'}\leq k$. Note that 
$$\gamma_{i_1}\circ\dots\circ\gamma_{i_m}\circ \phi=\phi\circ \sigma=\gamma_{j_1}\circ\dots\circ\gamma_{j_{m'}}\circ \phi.$$
Then $\gamma_{i_1}\circ\dots\circ\gamma_{i_m}$ and $\gamma_{j_1}\circ\dots\circ\gamma_{j_{m'}}$ are equal in the range of $\phi$.
Because $\phi$ is surjective, we have that $\gamma_{i_1}\circ\dots\circ\gamma_{i_m}=\gamma_{j_1}\circ\dots\circ\gamma_{j_{m'}}$. Thus the definition of $\tilde{\phi}$ does not depend on the choice of $i_1,\dots,i_m$. Moreover, $\tilde{\phi}$ is defined for every cell in $\tilde{N}$. Hence, $\tilde{\phi}$ is well-defined.

By definition $\tilde{\phi}(Id_N)=Id_Q$. Let $\gamma\neq Id_Q$ be a cell in $\tilde{Q}$. Then there are $1\leq i_1,\dots,i_m\leq k$ such that $\gamma=\gamma_{i_1}\circ\dots\circ\gamma_{i_m}$. Since $\tilde{\phi}(\sigma_{i_1}\circ\dots\circ\sigma_{i_m})=\gamma$, we have that $\tilde{\phi}$ is surjective.

From proposition~\ref{NRS1453}, the function $\tilde{\phi}$ is a network fibration if and only if
$\tilde{\phi}\circ \tilde{\sigma_i}=\tilde{\gamma_i}\circ \tilde{\phi}$, for every $i=1,\dots,k$.
Let $\sigma\neq Id_N$ be a cell in $\tilde{N}$. Then there are $1\leq i_1,\dots,i_m\leq k$ such that $\sigma=\sigma_{i_1}\circ\dots\circ\sigma_{i_m}$. 
For $1\leq i\leq k$, we have that $\tilde{\phi}\circ \tilde{\sigma_i}(Id_N)=\tilde{\gamma_i}\circ \tilde{\phi}(Id_N)$ and
\begin{align*}
\tilde{\phi}\circ \tilde{\sigma_i}(\sigma)= &\  \tilde{\phi}(\sigma_i\circ\sigma_{i_1}\circ\dots\circ\sigma_{i_m})=\gamma_i\circ\gamma_{i_1}\circ\dots\circ\gamma_{i_m}\\
= &\  \tilde{\gamma_i}(\gamma_{i_1}\circ\dots\circ\gamma_{i_m})=\tilde{\gamma_i}\circ \tilde{\phi}(\sigma_{i_1}\circ\dots\circ\sigma_{i_m})\ = \tilde{\gamma_i}\circ \tilde{\phi}(\sigma).
\end{align*} 
Hence $\tilde{\phi}$ is a surjective network fibration. By proposition~\ref{quosurjnetfri},  $\tilde{Q}$ is a quotient network of $\tilde{N}$.
\end{proof}

Using that $\tilde{N}=\tilde{\tilde{N}}$ (theorem~\ref{NRS1464}) and proposition~\ref{fundprelift}, we have the following. 

\begin{coro}
If $N$ is a quotient network of $L$ and $L$ is a quotient network of $\tilde{N}$, then $\tilde{N}=\tilde{L}$.
\end{coro}

\begin{obs}
From the proof of proposition~\ref{fundprelift}, it also follows that if $\phi:N\rightarrow Q$ is a surjective network fibration, then there exists a surjective network fibration $\tilde{\phi}:\tilde{N}\rightarrow \tilde{Q}$ such that for every cell $c$ in $N$ the following diagram is commutative
$$\begin{tikzpicture}
\node (n1)  {$\tilde{N}$};
\node (n2)  [right=of n1] {$\tilde{Q}$};
\node (n3)  [below=of n1] {$N$};
\node (n4)  [right=of n3] {$Q$};

\draw[->] (n1) -- (n2) node [midway,above] {$\tilde{\phi}$};
\draw[->] (n2) -- (n4) node [midway,right] {$\varphi^{Q}_{\phi(c)}$};
\draw[->] (n1) -- (n3) node [midway,left] {$\varphi^{N}_{c}$};
\draw[->] (n3) -- (n4) node [midway,below] {$\phi$};
\end{tikzpicture}
$$
where $\varphi^{Q}_{\phi(c)}$ and $\varphi^{N}_{c}$ are given by theorem~\ref{NRS1464}.
\end{obs}

The next example illustrates the fact that if a $S$ is a subnetwork of $N$ that does not implies the same relation between the corresponding fundamental network.In fact, we see that the existence of a (injective) network fibration $\phi:S\rightarrow N$ does not imply the existence of a network fibration $\tilde{\phi}:\tilde{S}\rightarrow \tilde{N}$.

\begin{exe}
Let $N$ be the network in figure~\ref{fig:hnai3} and $S$ the network in figure~\ref{fig:subhnai3}. The corresponding fundamental networks, $\tilde{N}$ and $\tilde{S}$, are given in figure~\ref{fig:fundhnai3} and \subref{fig:fundsubhnai3}.
There is an injective network fibration from $S$ to $N$, since $S$ is a subnetwork of $N$. However there is not an injective network fibration from $\tilde{S}$ to $\tilde{N}$, because $\tilde{S}$ is not a subnetwork of $\tilde{N}$. Moreover, there is not a network fibration from $\tilde{S}$ to $\tilde{N}$.
\end{exe}

In the following proposition, we show that the fundamental network construction transforms the subnetwork relation in the quotient network relation.

\begin{prop}\label{subfundquo}
Let $N$ be a homogeneous network with asymmetric inputs.
If $S$ is a subnetwork of $N$, then  $\tilde{S}$ is a quotient network of $\tilde{N}$.
\end{prop}

\begin{proof}
Let $N$  be a homogeneous network with asymmetric inputs and $S$ a subnetwork of $N$.
Suppose that $N$ is represented by the functions $\left(\sigma_i\right)_{i=1}^{k}$. 
Then $S$ is represented by the functions $\left(\left.\sigma_i\right|_{S}\right)_{i=1}^{k}$.

Consider the function $\tilde{\phi}:\tilde{N}\rightarrow \tilde{S}$ such that $\tilde{\phi}(\sigma)=\left.\sigma\right|_{S}$. This function is surjective, because if $\gamma=\left.\sigma_{i_1}\right|_{S}\circ\dots\circ\left.\sigma_{i_m}\right|_{S}$, then $\gamma=\left.(\sigma_{i_1}\circ\dots\circ\sigma_{i_m})\right|_{S}$.
For every cell $\sigma$ in $\tilde{N}$, we have that
$$\tilde{\phi}\circ \tilde{\sigma}_i(\sigma)=\tilde{\phi}(\sigma_i\circ\sigma)= \left.(\sigma_i\circ\sigma)\right|_{S}=\left.\sigma_i\right|_{S}\circ\left.\sigma\right|_{S}=\widetilde{\left.\sigma_i\right|_{S}}\circ \tilde{\phi}(\sigma).$$ 
Hence $\tilde{\phi}$ is a surjective network fibration. By proposition~\ref{quosurjnetfri}, it follows that $\tilde{S}$ is a quotient network of $\tilde{N}$.
\end{proof}

\subsection{Fundamental networks and lifts}\label{fundliftsec}

In this section, we give a characterization of the fundamental networks that are lifts of the original network, in terms of network connectivity, using the results in \cite{NRS14}.
We point out that Nijholt et al. in \cite{NRS14} consider that $N'$ is a quotient network of $N$ simply if there is a network fibration from $N$ to $N'$ which need not be surjective.
We also give a necessary condition for a network to be a lift of its fundamental network.

\begin{prop}\label{liftfund}
Let $N$ be a homogeneous network with asymmetric inputs and $\tilde{N}$ its fundamental network. Then $\tilde{N}$ is a lift of $N$ if and only if $N$ is backward connected.
\end{prop}

\begin{proof}
Let $N$ be a homogeneous network with asymmetric inputs and $\tilde{N}$ its fundamental network.
By proposition~\ref{quosurjnetfri}, the fundamental network $\tilde{N}$ is a lift of $N$ if and only if there is a surjective network fibration from $\tilde{N}$ to $N$.

Using theorem~\ref{NRS1464} and the network fibrations defined there, we have that every network fibration from $\tilde{N}$ to $N$ is equal to $\varphi_c$, for some cell $c$. Moreover, $\varphi_c$ is surjective if and only if $N_{(c)}=N$.
Note that $N_{(c)}=N$ if and only if $N$ is backward connected for $c$.
Hence $\tilde{N}$ is a lift of $N$ if and only if  $N$ is backward connected.
\end{proof}

It follows from proposition~\ref{subfundquo} and proposition~\ref{liftfund} that a fundamental network is a lift of every backward connected subnetwork of the original network. 

\begin{coro}\label{backsub}
Let $N$ be a homogenous network with asymmetric inputs, $\tilde{N}$ its fundamental network and $B$ a backward connected subnetwork of $N$. Then $\tilde{N}$ is a lift of $B$.
\end{coro}

%

In the next result, we give a necessary condition for a network to be a lift of its fundamental network.

\begin{prop}\label{surjtrans}
Let $N$ be a homogenous network with asymmetric inputs and $\tilde{N}$ its fundamental network.
If $N$ is a lift of $\tilde{N}$, then $N$ is transitive.
\end{prop}

\begin{proof}
Let $N$ be a homogenous network with asymmetric inputs and $\tilde{N}$ its fundamental network. Suppose that $N$ is a lift of $\tilde{N}$.
By proposition~\ref{quosurjnetfri}, there exists a surjective network fibration $\psi:N\rightarrow \tilde{N}$. Let $c$ be a cell in $N$ such $\psi(c)=Id_N$.
Consider the network fibrations, given in theorem~\ref{NRS1464}, $\varphi_d:\tilde{N}\rightarrow N$, for every cell $d$ in $N$. Note that $\varphi_d\circ \psi(c)=\varphi_d(Id_N)=d$, for every cell $d$ in $N$. Hence $N$ is transitive for $c$.
\end{proof}

\subsection{Fundamental networks and subnetworks}\label{fundsubsec}

In this section, we give a necessary and sufficient condition for a network to be a subnetwork of its fundamental network. Moreover, we give a sufficient condition for a fundamental network to be a subnetwork of the original network. We start with two examples.

\begin{exe}
(i) The network in figure~\ref{fig:hnai3} is not a subnetwork of its fundamental network, figure~\ref{fig:fundhnai3}. (ii) The network in figure~\ref{fig:subnet} is a subnetwork of its fundamental network, figure~\ref{fig:fundsubnet}.
\end{exe}

\begin{figure}[h]
\begin{subfigure}[t]{0.45\textwidth}
\centering
\begin{tikzpicture}
\node (n1) [draw, shape=circle] {1};
\node (n2) [draw, shape=circle] [right=of n1] {2};

\draw[-latex, thick] (n2) to  (n1);
\draw[-latex, thick] (n1) to  (n2);

\draw[-latex, dashed] (n1) to [loop above] (n1);
\draw[-latex, dashed] (n1) to [bend right] (n2);
\end{tikzpicture}
\caption{}
\label{fig:subnet}
\end{subfigure}
\begin{subfigure}[t]{0.45\textwidth}
\centering
\begin{tikzpicture}
\node (n1) [draw, shape=circle,label=center:$\sigma_2$] {\phantom{0}};
\node (n2) [draw, shape=circle,label=center:$\gamma$] [right=of n1] {\phantom{0}};
\node (n3) [draw, shape=circle,label=center:$Id$] [below=of n1]{\phantom{0}};
\node (n4) [draw, shape=circle,label=center:$\sigma_1$] [right=of n3] {\phantom{0}};

\draw[-latex, dashed] (n1) to [loop left] (n1);
\draw[-latex, dashed] (n1) to [bend right] (n2);
\draw[-latex, dashed] (n1) to (n3);
\draw[-latex, dashed] (n1) to (n4);

\draw[-latex, thick] (n2) to (n1);
\draw[-latex, thick] (n1) to (n2);
\draw[-latex, thick] (n4) to (n3);
\draw[-latex, thick] (n3) to (n4);
\end{tikzpicture}
\caption{$\gamma=\sigma_1\circ \sigma_2$}
\label{fig:fundsubnet}
\end{subfigure}
\caption{\subref{fig:subnet} Homogeneous network with asymmetric inputs represented by $\sigma_1=[2\;1]$ and $\sigma_2=[1\;1]$; \subref{fig:hnai2} Fundamental network of the network \subref{fig:subnet}.}
\label{fig::subnet}
\end{figure}
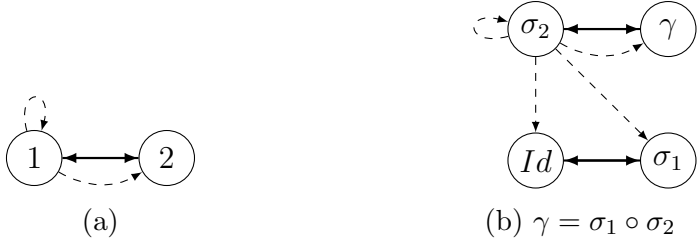

In the next proposition, we give necessary and sufficient conditions for the existence of a network fibration from a network to its fundamental network.

\begin{prop}\label{graphfricond}
Let $N$ be a  homogeneous network with asymmetric inputs and $\tilde{N}$ its fundamental network with sets of cells $C$ and $\tilde{C}$, respectively. Suppose that $N$  is backward connected for $c\in C$.\\
(i) If $\varphi:N\rightarrow\tilde{N}$ is a network fibration, then $\sigma'\circ \varphi(c) =\sigma''\circ \varphi(c)$, for every $\sigma',\sigma''\in \tilde{C}$ such that $\sigma'(c)=\sigma''(c)$. \\
(ii) If there is $ \sigma\in \tilde{C}$ such that $\sigma'\circ \sigma =\sigma''\circ \sigma$,  for every $\sigma',\sigma''\in \tilde{C}$ such that $\sigma'(c)=\sigma''(c)$, then there is a network fibration $\varphi:N\rightarrow\tilde{N}$ such that $\varphi(c)=\sigma$.
\end{prop}

\begin{proof}
Let $N$ be a  homogeneous network with asymmetric inputs and $\tilde{N}$ its fundamental network with sets of cells $C$ and $\tilde{C}$, and represented by $\left(\sigma_i\right)_{i=1}^{k}$ and $\left(\tilde{\sigma_i}\right)_{i=1}^{k}$, respectively. Suppose that $N$  is backward connected for $c\in C$.

In order to prove $(i)$, suppose that $\varphi:N\rightarrow\tilde{N}$ is a network fibration. By proposition~\ref{NRS1453}, $\varphi\circ \sigma_i=\tilde{\sigma_i}\circ \varphi=\sigma_i\circ \varphi$, for every $1\leq i\leq k$. So for every $\sigma\in \tilde{C}$, we have that $$\varphi\circ \sigma=\sigma\circ \varphi.$$
Let $\sigma',\sigma''\in \tilde{C}$ such that $\sigma'(c)=\sigma''(c)$. Then
$$\sigma'\circ \varphi(c)=\varphi\circ \sigma'(c)=\varphi\circ \sigma''(c)=\sigma''\circ \varphi(c).$$

To prove $(ii)$, suppose that there is $ \sigma\in \tilde{C}$ such that $\sigma'\circ \sigma =\sigma''\circ \sigma$,  for every $\sigma',\sigma''\in \tilde{C}$ such that $\sigma'(c)=\sigma''(c)$.
Define $\varphi:N\rightarrow \tilde{N}$ given by $\varphi(c)=\sigma$ and $\varphi(c')=\sigma'\circ\sigma$, where $c'=\sigma'(c)$. This function is defined for every cell in $N$, because $N$ is backward connected for $c$. And it is well defined, because if $c'=\sigma'(c)=\sigma''(c)$, then $\varphi(c')=\sigma'\circ \sigma=\sigma''\circ \sigma$.

We just need to see that $\varphi$ is a network fibration. Using proposition~\ref{NRS1453}, we check that $\varphi\circ \sigma_i=\tilde{\sigma_i}\circ \varphi$, for every $1\leq i \leq k$. Because $N$ is backward connected, for every cell $d$ of $N$, there is $\sigma'\in \tilde{C}$ such that $d=\sigma'(c)$ and
$$\varphi\circ \sigma_i(d)=\varphi(\sigma_i\circ\sigma'(c))=\sigma_i\circ\sigma'\circ \sigma=\tilde{\sigma_i}(\sigma'\circ \sigma)=
\tilde{\sigma_i}\circ \varphi(\sigma'(c))=\tilde{\sigma_i}\circ \varphi(d),$$
for every $1\leq i \leq k$. Hence $\varphi\circ \sigma_i=\tilde{\sigma_i}\circ \varphi$ and $\varphi$ is a network fibration.
\end{proof}

Recalling proposition~\ref{subinjnetfri} and restricting the network fibration of proposition~\ref{graphfricond} to an injective network fibration, we obtain the characterization of the networks that are subnetworks of its fundamental network.

\begin{coro} Let $N$ be a homogeneous network with asymmetric inputs backward connected for a cell $c$ and $\tilde{N}$ its fundamental network. Then  $N$ is a subnetwork of $\tilde{N}$ if and only if there is $\sigma\in \tilde{C}$ such that for every $\sigma',\sigma''\in \tilde{C}$, the following condition is satisfied: $$\sigma'\circ \sigma =\sigma''\circ \sigma\Leftrightarrow\sigma'(c)=\sigma''(c).$$
\end{coro}

\begin{exe}
Consider the network in figure~\ref{fig:subnet} represented by $\sigma_1=[2\;1]$ and $\sigma_2=[1\;1]$. The network is backward connected for the cell $1$ and $\sigma'\circ \sigma_2 =\sigma''\circ \sigma_2$ if and only if $\sigma'(1)=\sigma''(1)$. By the previous corollary, the network is a subnetwork of its fundamental network.
\end{exe}

We show now that if a network is transitive, then its fundamental network is a subnetwork of the network. This result will be used in the following section to characterize fundamental networks.

\begin{prop}\label{transinj}
Let $N$ be a homogenous network with asymmetric inputs and $\tilde{N}$ its fundamental network.
If $N$ is transitive, then $\tilde{N}$ is a subnetwork of $N$.
\end{prop}

\begin{proof}
Let $N$ be a homogenous network with asymmetric inputs and $\tilde{N}$ its fundamental network.
 Denote the network fibrations, given in theorem~\ref{NRS1464}, by $\varphi_d:\tilde{N}\rightarrow N$, for every cell $d$ in $N$. Suppose that $N$ is transitive for a cell $c$.
Then for every cell $d$ in $N$ there is a network fibration $\psi_d:N\rightarrow N$ such that $\psi_d(c)=d$. In order to prove that $\tilde{N}$ is a subnetwork of $N$, we show that $\varphi_c$ is an injective network fibration.

Note that $\psi_d\circ \varphi_c(Id)=\psi_d(c)=d=\varphi_d(Id)$. By propositions~\ref{backconfund} and \ref{netfriback}, we have that $\psi_d\circ \varphi_c=\varphi_d$.
If $\varphi_c(\gamma_1)=\varphi_c(\gamma_2)$, then for every cell $d$ in $N$
$$\gamma_1(d)=\varphi_d(\gamma_1)=\psi_d\circ\varphi_c(\gamma_1)=\psi_d\circ\varphi_c(\gamma_2)=\varphi_d(\gamma_2)=\gamma_2(d),$$
and $\gamma_1=\gamma_2$. Hence $\varphi_c$ is an injective network fibration. By proposition~\ref{subinjnetfri}, $\tilde{N}$ is a subnetwork of $N$.
\end{proof}

From propositions~\ref{surjtrans} and \ref{transinj}, we have the following result.

\begin{coro}
Let $N$ be a homogenous network with asymmetric inputs and $\tilde{N}$ its fundamental network.
If $N$ is a lift of $\tilde{N}$, then $\tilde{N}$ is a subnetwork of $N$.
\end{coro}

\subsection{Networks which are fundamental networks}\label{fundchasec}

Using theorem~\ref{NRS1464} and the results obtained in the previous sections, we can now characterize the networks that are fundamental networks, in terms of transitivity and backward connectedness. 

\begin{teo}
Let $N$ be a homogeneous network with asymmetric inputs. 
The network $N$ is a fundamental network if and only if there is a cell $c$ such that $N$ is backward connected for $c$ and transitive for $c$.
\end{teo}

\begin{proof}
Let $N$ be a homogeneous network with asymmetric inputs. 

Suppose that $N$ is a fundamental network. Then $N$ is equal to $\tilde{N}$ and there is a bijective network fibration $\psi:\tilde{N}\rightarrow N$.
From proposition~\ref{backconfund}, we know that $\tilde{N}$ is backward connected for $Id$.
By theorem~\ref{NRS1464}, we have for every cell $\sigma$ in $\tilde{N}$ that there is a network fibration   $\phi_\sigma:\tilde{N}\rightarrow \tilde{\tilde{N}}=\tilde{N}$ such that $\phi_\sigma(\gamma)=\gamma\circ \sigma$. In particular $\phi_\sigma(Id)=\sigma$, and $\tilde{N}$ is transitive for $Id$.
Hence, $N$ is backward connected for $\psi(Id)$ and it is transitive for $\psi(Id)$.

Suppose that there is a cell $c$ in $N$ such that $N$ is backward connected for $c$ and transitive for $c$. We show that $N$ is equal to $\tilde{N}$ by showing that there is a bijective network fibration from $\tilde{N}$ to $N$. In fact, the network fibration $\varphi_c:\tilde{N}\rightarrow N$, given by theorem~\ref{NRS1464}, is a bijection, since it is surjective by proposition~\ref{liftfund}, and it is injective by proposition~\ref{transinj}.
\end{proof}

\section{Architecture of networks: rings and depth}\label{arcsec}

In this section, we introduce the definition of rings and depth of a homogenous network with asymmetric inputs.
We start by recalling the definitions of connected and strongly connected components. We finish by describing how we can obtain the rings and the depth of a homogenous network with asymmetric inputs using the representative functions of the network.

We say that there is an \emph{undirected path} in a network connecting the sequence of cells $(c_0,c_1,\dots,c_{k-1},c_k)$, if for every $j=1,\dots,k$ there is an edge from $c_{j-1}$ to $c_j$ or an edge from $c_j$ to $c_{j-1}$. A directed path $(c_0,c_1,\dots,c_{m-1},c_k)$ is called a \emph{cycle}, if $c_0=c_k$.

\begin{defi}
Let $N$ be a network.
A subset $Y$ of cells in $N$ is called \emph{connected} if for every two cells in $Y$ there is an undirected path between them.

We say that $Y$ is a \emph{connected component} of $N$, if $Y$ is a maximal connected subset of cells, in the sense that if $Y\cup \{c\}$ is connected then $c\in Y$.
\end{defi}

We can partition the set of cells of a network in its connected components.

\begin{defi}
Let $N$ be a network with set of cells $C$ and a subset $X\subseteq C$.\\
(i)   The cells $c_1, c_2\in C$ are \emph{strongly connected}, if there is a directed path from $c_1$ to $c_2$ and a directed path from $c_2$ to $c_1$.\\
(ii)  The subset $X$ is \emph{strongly connected}, if every $c_1, c_2\in X$ are strongly connected.\\
(iii) The subset $X$ is a \emph{strongly connected component} of $N$, if $X$ is a maximal strongly connected subset of cells.\\
(iv)  The subset $X$ is a \emph{source} of $N$, if $X$ is a strongly connected component that does not receive any edge with source cell outside $X$, i.e., $s(I(X))\subseteq X$.
\end{defi}


Let $N$ be a homogeneous network with asymmetric inputs and $i$ an edge type of $N$. Denote by $N_i$ the network with the same cells of $N$ and only the edges of type $i$ of $N$. Let $C_i^1,\dots,C_i^m$ be the partition of the set of cells of the network $N_i$ in its connected components. For each connected component, the topology of $N_i$ is the union of a unique source component and feed-forward networks starting at some cell of the source component. See figure~\ref{fig::exarchnet} for an example and see \cite[{proposition 2.3}]{G14} for details.
For each $j=1,\dots,m$, we call the source of $N_i$ in $C_i^j$ a \emph{ring} and denote it by $R_{i}^{j}$. Since the cells in the network $N_i$ have only one input, every cycle in $N_i$ connects every cell in a ring.

\begin{figure}[h]
\centering
\begin{tikzpicture}
\node (n1) [shape=circle,draw] 							 {\phantom{0}};
\node (n2) [shape=circle,draw] [right=of n1] {\phantom{0}};
\node (n3) [shape=circle,draw] [right=of n2] {\phantom{0}};
\node (n4) [shape=circle,draw] [below=of n1] {\phantom{0}};
\node (n5) [shape=circle,draw] [right=of n4] {\phantom{0}};
\node (n6) [shape=circle,draw] [right=of n5] {\phantom{0}};
\node (n7) [shape=circle,draw] [below=of n4] {\phantom{0}};
\node (n8) [shape=circle,draw] [right=of n7] {\phantom{0}};
\node (n9) [shape=circle,draw] [right=of n8] {\phantom{0}};

\draw[-latex, thick] (n2) to  (n1);
\draw[-latex, thick] (n5) to  (n2);
\draw[-latex, thick] (n6) to  (n3);
\draw[-latex, thick] (n7) to  (n4);
\draw[-latex, thick] (n4) to  (n5);
\draw[-latex, thick] (n8) to  (n6);
\draw[-latex, thick] (n5) to  (n7);
\draw[-latex, thick] (n7) to  (n8);
\draw[-latex, thick] (n8) to  (n9);
\end{tikzpicture}
\caption{Union of a ring and feed-forward networks starting at the ring.}
\label{fig::exarchnet}
\end{figure}
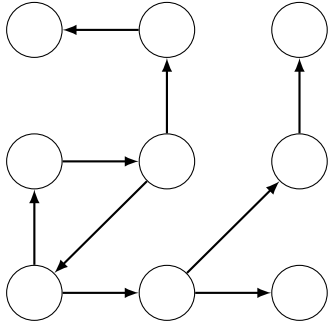

\begin{defi}
Let $N$ be a homogeneous network with asymmetric inputs and $i$ an edge type of $N$. Let $C_i^1,\dots,C_i^m$ be the connected components of $N_i$. For each connected component, $C_i^j$, of $N_i$, we define the \emph{depth} of $N_i$ in $C_i^j$ by
$$\depth^j_i(N):=\max \{ \min\{|(r,c)|:r\in R_{i}^{j}\}: c\in C_i^j\},$$
where $|(r,c)|$ is $0$, if $r=c$, or the number of edges in the shortest directed path in $N_i$ from $r$ to $c$.
And the \emph{depth} of $N_i$ is
\[\depth_i(N):=\max_{j=1,\dots,m}\{\depth^j_i(N)\}.\qedhere\]
\end{defi}

\begin{figure}[h]
\begin{subfigure}[t]{0.32\textwidth}
\centering
\begin{tikzpicture}
\node (n1) [shape=circle,draw]  {1};

\draw[-latex, thick] (n1) to [loop left] (n1);
\end{tikzpicture}
\label{fig:hnai31}
\end{subfigure}
\begin{subfigure}[t]{0.32\textwidth}
\centering
\begin{tikzpicture}
\node (n2) [circle,draw]   {2};
\node (n3) [shape=circle,draw] [right=of n2] {3};
\draw[-latex, thick] (n2) to [loop left] (n2);
\draw[-latex, thick] (n2) to  (n3);
\end{tikzpicture}
\label{fig:hnai32}
\end{subfigure}
\begin{subfigure}[t]{0.32\textwidth}
\centering
\begin{tikzpicture}
\node (n4) [circle,draw]   {4};
\node (n5) [circle,draw]  [right=of n4] {5};
\draw[-latex, thick] (n5) to  (n4);
\draw[-latex, thick] (n4) to  (n5);
\end{tikzpicture}
\label{fig:hnai33}
\end{subfigure}
\caption{The restriction of the network in figure~\ref{fig:hnai3} to the solid edges has three connected components and its depth os $1$.On the left, the ring is $\{1\}$ and the depth is $0$. On the center, the ring is $\{2\}$ and the depth is $1$. On the right, the ring is $\{4,5\}$ and the depth is $0$.}
\label{fig::archhnai}
\end{figure}
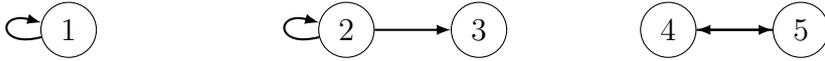

\begin{exe}\label{drexe}
Let $N$ be the network in figure~\ref{fig:hnai3}. Consider the restriction $N_1$ to the solid edges represented in figure~\ref{fig::archhnai}. The network $N_1$ has three connected components, $C^1_1=\{1\}$,  $C^2_1=\{2,3\}$ and $C^3_1=\{4,5\}$. The rings of $N_i$ are: $R^1_1=\{1\}$ in $C^1_1$; $R^2_1=\{2\}$ in $C^2_1$; and $R^3_1=\{4,5\}$ in $C^3_1$. The depth of $N_1$: in $C^1_1$ is $0$; in $C^2_1$ is $1$; and in $C^3_1$ is $0$. So the depth of $N_1$ is $1$.

Let $\tilde{N}$ be the fundamental network of $N$ represented in figure~\ref{fig:fundhnai3}. Consider the restriction $\tilde{N}_1$ to the solid edges. The network $\tilde{N}_1$ has four connected components. Each of the connected components has a ring of size $2$. And the depth of $\tilde{N}_1$ is $1$. Note that the size of any ring in $\tilde{N}_1$ is a multiple of the size of some rings in $N_1$ and the depth of $N_1$ is equal to  the depth of $\tilde{N}_1$. In the next section, we formalize and prove these observations to every networks with asymmetric inputs.
\end{exe}

We describe now the rings and the depth of a network using representative functions. This follows from the following facts: every representative function, $\sigma_i$, is semi-periodic, i.e., there exist $a\geq 0$ and $b> 0$ such that $\sigma_i^a=\sigma_i^{a+b}$; if $\sigma_i^a=\sigma_i^{a+b}$, then there is a cycle for every cell in the range of $\sigma_i^a$; and the distance of a cell $c$ to  a ring $R$ is equal to the minimum $p\geq 0$ such that $\sigma_i^p(c)\in R$.

\begin{lem}\label{dlhnai}
Let $N$ be a homogeneous network with asymmetric inputs represented by the functions $\left(\sigma_i\right)_{i=1}^{k}$ and $C$ the set of cells of $N$. Fix $1\leq i\leq k$ and denote the connected components of $N_i$ by $C^1_i,\dots,C^m_i$, and the corresponding rings by $R_{i}^{1},\dots, R_{i}^{m}$.\\
(i) If 
	$\sigma_i^a=\sigma_i^{a+b}$ for some  $a\geq 0$ and $b>0$, then $R_{i}^{j}=\sigma_i^a(C^j_i)$ for $1\leq j\leq m$.\\
(ii) $$\depth_i(N)=\min\left\{p\in\mathbb{N}_0: \sigma_i^p(C)\subseteq \bigcup_{j=1}^{m}R_{i}^{j}\right\}.$$
\end{lem}

%
%
%

\begin{exe}
Consider example~\ref{drexe}. Let $N$ be the network in figure~\ref{fig:hnai3}, $N_1$ its restriction to the solid edges represented by the function $\sigma_1=[1\;2\;2\;5\;4]$ and $C^1_1=\{1\}$,  $C^2_1=\{2,3\}$ and $C^3_1=\{4,5\}$ the connected components of $N_1$ appearing in figure~\ref{fig::archhnai}. 
Note that $\sigma_1=\sigma^3_1$. By lemma~\ref{dlhnai}, the rings of $N_1$ are $R_1^1=\sigma_1(C^1_1)=\{1\}$, $R_1^2=\sigma_1(C^2_1)=\{2\}$, and $R_1^3=\sigma_1(C^3_1)=\{4,5\}$. Moreover, $\sigma_1^k(C_1^1)\subseteq R_1^1\cup R_1^2\cup R_1^3$ if and only if $k\geq 1$. Hence $\depth_1(N)=1$.
\end{exe}

\section{Architecture of fundamental networks}\label{strufundsec}

We start this section by studying the connectivity of fundamental networks for which the semi-group generated by their representative functions is in fact a group. 

\begin{prop}
Let $N$ be a homogenous network with asymmetric inputs and $\tilde{N}$ its fundamental network. 
\begin{itemize}
\item[(a)] The following statements are equivalent:
\begin{itemize}
\item[(i)] $\tilde{N}$ is strongly connected.
\item[(ii)] $\tilde{C}$ is a group.
\item[(iii)] The representative functions of $N$ are bijections (i.e., permutations).
\end{itemize} 
\item[(b)] If $N$ is connected and $\tilde{N}$ is strongly connected, then $N$ is strongly connected.
\end{itemize}
\end{prop}

\begin{proof}
Let $N$ be a homogenous network with asymmetric inputs and $\tilde{N}$ its fundamental network with set of cells $C$ and $\tilde{C}$, respectively.

If $\tilde{N}$ is strongly connected, then there is a directed path between every pair of cells in $\tilde{C}$, in particular, between $Id$ and $\sigma\in \tilde{C}$. Thus
$$\forall_{\sigma\in \tilde{C}}\  \exists_{\sigma'\in \tilde{C}}\  : \sigma' \circ \sigma=Id,$$ 
where $\sigma'$ is a directed path from $Id$ to $\sigma$. Conversely, if $\tilde{C}$ is a group, then there is a directed path between every pair of cells in $\tilde{C}$.
This proves that $(i)$ is equivalent to $(ii)$. 

Any representative function is invertible if and only if it is a bijection. And every permutation has a finite order, i.e., exists $k$ such that $\sigma^k=Id$.
Hence the statements $(ii)$ and $(iii)$ are equivalent.

Now, to prove (b), suppose that $N$ is connected and $\tilde{N}$ is strongly connected. Then $\tilde{C}$ is a group and for every representative function $\sigma$ of $N$, there exist $\sigma^{-1}$. Note that $\sigma^{-1}$ is not always a representative function, but it is a composition of representative functions, by definition of $\tilde{C}$. We refer to $\sigma^{-1}$ has the inverse path of the connection $\sigma$. 
Moreover, for every two cells $c$ and $d$ there exists an undirected path from $c$ to $d$, because $N$ is connected. From this undirected path it is possible to get a directed path in $N$ from $c$ to $d$ by considering for each connection in the undirected path either the connection itself or its inverse path.
\end{proof}

\subsection{Depth of fundamental networks}\label{depthfundsec}

In example~\ref{drexe}, we presented a network such that the depth of the network is equal to the depth of its fundamental network. We prove now that this property is valid for every homogenous network with asymmetric inputs. Moreover, we use this fact to show that an adjacency matrix of a network is non-singular if and only if the correspondent adjacency matrix of its fundamental network is non-singular.

\begin{prop}\label{dfhnai}
Let $N$ be a homogeneous network with asymmetric inputs represented by the functions $\left(\sigma_i\right)_{i=1}^{k}$ and $\tilde{N}$ its fundamental network. Then $$\depth_i(N)=\depth_i(\tilde{N}),$$ where $i=1,\dots,k$.
\end{prop}

\begin{proof}
Let $N$ be a homogeneous network with asymmetric inputs represented by  $\left(\sigma_i\right)_{i=1}^{k}$, $C$ its set of cells and $\tilde{N}$ its fundamental network. Fix $1\leq i\leq k$.
Denote the connected components of $N_i$ by  $C^1_i,\dots,C^m_i$ and the corresponding rings by $R_{i}^{1},\dots, R_{i}^{m}$. Denote the connected components of $\tilde{N}_i$ by  $\tilde{C}^1_i,\dots,\tilde{C}^{\tilde{m}}_i$ and the corresponding rings by $\tilde{R}_{i}^{1},\dots, \tilde{R}_{i}^{\tilde{m}}$. Let  $p_i=\depth_i(N)$ and $\tilde{p_i}=\depth_i(\tilde{N})$.

By lemma~\ref{dlhnai}~(ii), we have that $\sigma_i^{p_i}(C)\subseteq R_i^{1}\cup\dots\cup R_i^{m}$. For each connected component $C^j_i$ of $N_i$. Since cycles of $N_i$ in $C^j_i$ have to start in a cell of $R_i^j$ and travel by the other cells in $R_i^j$ to reach the initial point, we have that $\sigma_i^{p_i}(C^j_i)=\sigma_i^{p_i+r}(C^j_i)$ if and only if $r$ is a multiple of $|R_{i}^{j}|$. 
Then $\sigma_i^{p_i}=\sigma_i^{p_i+r}$, if $r=\lcm\{|R_{i}^{1}|,\dots,|R_{i}^{k}|\}$, where $\lcm$ is the least common multiple.

Note that $\tilde{\sigma}_i^{p_i}=\tilde{\sigma}_i^{p_i+r}$, because $\tilde{\sigma}_i^{p_i}(\sigma)=\sigma_i^{p_i}\circ\sigma=\sigma_i^{p_i+r}\circ\sigma=\tilde{\sigma}_i^{p_i+r}(\sigma)$.  By lemma~\ref{dlhnai}~(i), 
$$\bigcup_{j=1}^{\tilde{m}}\tilde{R}_{i}^{j}=\bigcup_{j=1}^{\tilde{m}}\tilde{\sigma}_i^{p_i}(\tilde{C}^j_i)=\tilde{\sigma}_i^{p_i}(\tilde{C})$$
Hence $\tilde{p_i}\leq p_i$, by lemma~\ref{dlhnai}~(ii).

From $\tilde{\sigma}_i^{p_i}=\tilde{\sigma}_i^{p_i+r}$, we also know that $\sigma_i^{p_i},\dots,\sigma_i^{p_i+r-1}$ is a ring of $\tilde{N}_i$, because $(\sigma_i^{p_i},\dots,\sigma_i^{p_i+r-1},\sigma_i^{p_i+r}=\sigma_i^{p_i})$ is a cycle in $\tilde{N}_i$.
The directed path $Id=\sigma_i^0,\sigma_i^1,\dots,\sigma_i^{p_i-1},\sigma_i^{p_i}$ is the shortest directed path in $\tilde{N}_i$ from $Id$ to a cell in this ring. Then we have that $\tilde{p_i}\geq p_i$ and thus conclude that $\tilde{p_i}= p_i$.
\end{proof}

A network can be represented by its \emph{adjacency matrices} $A_i$, one for each edge type $i$. More precisely, if the network has $n$ cells, say $C=\left\{1,\dots,n \right\}$, then the matrix $A_i$ is an $n\times n$ matrix, where the entry $(A_i)_{c\:c'}$ denotes the number of edges of type $i$ from $c'$ to $c$. 

\begin{coro}
Let $N$ be a homogeneous network with asymmetric inputs and $\tilde{N}$ its fundamental network. Denote by $A_i$ the adjacency matrix of $N$ and $\tilde{A}_i$ the adjacency matrix of $\tilde{N}$, for an edge type $i$. 

Then $A_i$ is non-singular if and only if $\tilde{A}_i$ is non-singular.
\end{coro}

\begin{proof}
The eigenvalues of the adjacency matrix of a homogeneous network with asymmetric inputs for an edge of type $i$, $A_i$, are $1,w_j,w_j^2,\dots, w_j^{r_j-1}$ where $r_j= |R_{i}^{j}|$, $w_j=\exp^{2\pi\imath/r_j}$, $R_i^j$ is the ring of type $i$ of $N$ in $C_i^j$ and $C^1_i,\dots, C^m_i$ are the connected components of $N_i$ and $0$ if $\depth_i(N)\neq 0$. Hence $A_i$ is non-singular if and only if $\depth_i(N)= 0$ if and only if $\depth_i(\tilde{N})= 0$ if and only if $\tilde{A}_i$ is non-singular.
\end{proof}

\subsection{Rings of fundamental networks}\label{ringsfundsec}

We consider now the relation between the size of the rings in a network and of those in its fundamental network. Specifically, we show that the size of a ring in a fundamental network is a (least common) multiple of some ring's sizes in the network. Moreover we use this result to fully describe the fundamental network of a network with only one edge type.

\begin{prop}\label{sizeringfund}
Let $N$ be a homogeneous network with asymmetric inputs represented by the functions $\left(\sigma_i\right)_{i=1}^{k}$, $C$ the set of cells of $N$ and $\tilde{N}$ its fundamental network. Fix $1\leq i\leq k$. Let $C^1_i,\dots,C^m_i$ be the connected components of $N_i$ and $R_{i}^{1},\dots, R_{i}^{m}$ the corresponding rings. Analogously, let $\tilde{C}^1_i,\dots,\tilde{C}^{\tilde{m}}_i$ be the connected components of $\tilde{N}_i$ and $\tilde{R}_{i}^{1},\dots, \tilde{R}_{i}^{\tilde{m}}$ the corresponding rings. 
If $1\leq j\leq \tilde{m}$ and $ \gamma\in \tilde{C}^j_i$, then
$$|\tilde{R}_{i}^{j}|=\lcm\left\{|R_{i}^{j'}|: C^{j'}_i\cap\gamma(C) \neq \emptyset \right\}.$$
Moreover, there exists $1\leq j\leq \tilde{m}$ such that 
$|\tilde{R}_{i}^{j}|=\lcm\left\{|R_{i}^{1}|, \dots, |R_{i}^{m}|\right\}$.
\end{prop}

\begin{proof}
Let $N$ be a homogeneous network with set of cells $C$ and asymmetric inputs represented by the functions $\left(\sigma_i\right)_{i=1}^{k}$. Let $\tilde{N}$ be its fundamental network. Fix $1\leq i\leq k$. Let $C^1_i,\dots,C^m_i$ be the connected components of $N_i$ and $R_{i}^{1},\dots, R_{i}^{m}$ the corresponding rings. Analogously, let $\tilde{C}^1_i,\dots,\tilde{C}^{\tilde{m}}_i$ be the connected components of $\tilde{N}_i$ and $\tilde{R}_{i}^{1},\dots, \tilde{R}_{i}^{\tilde{m}}$ the corresponding rings. Let $p_i=\depth_i(N)=\depth_i(\tilde{N})$. 
Choose $j$ and $\gamma$ such that $1\leq j\leq \tilde{m}$ and $ \gamma\in \tilde{C}^j_i$.
 Define $J=\{j': \gamma(C)\cap C_i^{j'}\neq \emptyset\}$, $r^{\gamma}= \lcm\{|R_i^{j'}|:j'\in J|\}$ and $C^{\gamma}=\cup_{j'\in J}C_i^{j'}$.

By lemma~\ref{dlhnai}, $$\sigma^{p_i}_i(C^{\gamma})=\bigcup_{j'\in J}R_i^{j'}.$$
Note that $\left.\sigma^{p_i}_i\right|_{C^{\gamma}}= \left.\sigma^{p_i+r^{\gamma}}_i\right|_{C^{\gamma}}$, because $r^{\gamma}$ is a multiple of $|R_i^{j'}|$, for every $j'\in J$. Then $\tilde{\sigma}^{p_i}_i \circ \gamma =\sigma^{p_i}_i \circ \gamma = \sigma^{p_i+r^{\gamma}}_i\circ \gamma =\tilde{\sigma}^{p_i+r^{\gamma}}_i\circ \gamma$ and $(\sigma_i^{p_i}\circ\gamma, \dots, \sigma^{p_i+r^{\gamma}}_i\circ\gamma)$ is a cycle in $\tilde{N}_i$. Since $\gamma\in \tilde{C}^j_i$, we have that $\sigma_i^{p_i}\circ \gamma, \dots, \sigma^{p_i+r^{\gamma}-1}_i\circ\gamma\in \tilde{C}^j_i$ and so the ring of $\tilde{N}_i$ in $\tilde{C}^j_i$ is $\tilde{R}_i^{j}=\{\sigma^{p_i+1}_i\circ\gamma,\dots, \sigma^{p_i+r^{\gamma}}_i\circ\gamma\}$. This cycle does not repeat cells, because $r^\gamma$ is the least common multiple. Thus 
$$|\tilde{R}_i^{j}|= r^{\gamma}=\lcm\left\{|R_{i}^{j'}|: C^{j'}_i\cap\gamma(C) \neq \emptyset\right\}.$$

The second part of the result follows from taking $\gamma=Id_C$.
\end{proof}

Propositions~\ref{dfhnai} and \ref{sizeringfund} can be used to describe the fundamental network of a homogenous network with only one edge type. 

\begin{defi}[{\cite[definition 3.1.]{M14}, \cite[definition 2.4]{G14}}]
Let $N$ be a homogeneous network with asymmetric inputs that has only one edge type. We say that $N$ is a \emph{loop-chain} with size $l\geq 1$ and $p\geq 0$, if $N$ has $l+p$ cells, it has a unique source component with $l$ cells and the depth of $N$ is $p$. 
\end{defi}

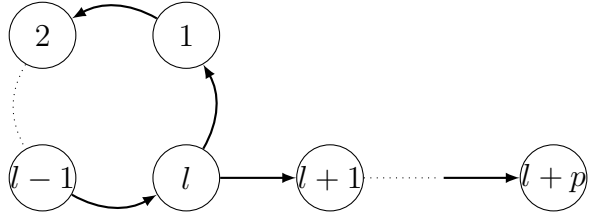
\begin{figure}[h]
\center
\begin{tikzpicture}
\node (n1) [circle,draw,label=center:$1$]   {\phantom{00}};

\node (n2) [circle,draw,label=center:$2$] [left=of n1] {\phantom{00}};
\node (n3) [circle,draw,label=center:$l-1$] [below=of n2] {\phantom{00}};
\node (n4) [circle,draw,label=center:$l$] [right=of n3] {\phantom{00}};
\node (n5) [circle,draw,label=center:$l+1$] [right=of n4] {\phantom{00}};
\node (n6)  [right=of n5] {\phantom{\hspace{-5mm}}};
\node (n7) [circle,draw,label=center:$l+p$] [right=of n6] {\phantom{00}};

\draw[-latex, thick] (n1) to [bend right] (n2);
\draw[dotted] (n2) to [bend right] (n3);
\draw[-latex, thick] (n3) to [bend right] (n4);
\draw[-latex, thick] (n4) to [bend right] (n1);
\draw[-latex, thick] (n4) to  (n5);
\draw[dotted] (n5) to (n6);
\draw[-latex, thick] (n6) to (n7);
\end{tikzpicture}
\caption{The fundamental network of a homogenous network with asymmetric inputs $N$ having only one edge type is a loop-chain with size $l$ and $p$, where $l$ is the least common multiple of all ring's sizes in $N$ and $p$ is the depth of $N$.}
\label{fund1input}
\end{figure}

\begin{coro}
Let $N$ be a homogeneous network with asymmetric inputs and only one edge type. If $l$ is the least common multiple of the size of all the rings in $N$ and $p$ is the depth of $N$, then the fundamental network of $N$ is a loop-chain with size $l$ and $p$.
\end{coro}

\begin{proof}
Let $N$ be a homogeneous network with asymmetric inputs that has only one edge type, $l$ the least common multiple of the size of all the rings in $N$, $p$ the depth of $N$ and $\tilde{N}$ its fundamental network. 

We know by proposition~\ref{backconfund} that $\tilde{N}$ is backward connected and so $\tilde{N}$ has only one connected component. The size of the ring of that connected component is equal to the least common multiple of the sizes of rings in $N$, see proposition~\ref{sizeringfund}. By proposition~\ref{dfhnai}, we also know that $\depth(N)=\depth(\tilde{N})$. 
Then $\tilde{N}$ has at least the loop-chain with size $l$ and $p$ described in figure~\ref{fund1input}. 


Next, we prove that $\tilde{N}$ has only $l+p$ cells. Suppose that there exists more than $l+p$ cells. Then there is a cell $j>l+p$ that receives an edge from the cells $1,\dots,l+p$, because $\tilde{N}$ has only one connected component and the first $l+p$ cells already receive an edge from the first $l+p$ cells. If $j$ receives an edge from the cells $1,\dots,l+p-1$, then $\tilde{N}$ is not backward connected. If $j$ receives an edge from the cell $l+p$, then $\depth(\tilde{N})>p$. Hence $\tilde{N}$ is a loop-chain with size $l$ and $p$ described in figure~\ref{fund1input}.
\end{proof}


\begin{thebibliography}{10}
\bibitem{ADGL09}
M.~Aguiar, A.~Dias, M.~Golubitsky, and M.~Leite.
\newblock Bifurcations from regular quotient networks: a first insight.
\newblock {\em Phys. D} {\bf 238(2)} (2009), 137--155.

\bibitem{B74}
N.~Biggs.
\newblock {\em Algebraic graph theory}.
\newblock Cambridge Tracts in Mathematics, No. 67
\newblock (London: Cambridge University Press, 1974).

\bibitem{BV02}
P.~Boldi and S.~Vigna.
\newblock Fibrations of graphs.
\newblock {\em Discrete Math.} {\bf 243(1-3)} (2002), 21--66.

\bibitem{DL15}
L.~DeVille and E.~Lerman.
\newblock Modular dynamical systems on networks.
\newblock {\em J. Eur. Math. Soc. (JEMS)} {\bf 17(12)} (2015), 2977--3013.

\bibitem{G14}
A.~Ganbat.
\newblock Reducibility of steady-state bifurcations in coupled cell systems.
\newblock {\em J. Math. Anal. Appl.} {\bf 415(1)} (2014), 159--177.

\bibitem{GST05}
M.~Golubitsky, I.~Stewart, and A.~T{\"o}r{\"o}k.
\newblock Patterns of synchrony in coupled cell networks with multiple arrows.
\newblock {\em SIAM J. Appl. Dyn. Syst.} {\bf 4(1)} (2005), 78--100.

\bibitem{K09}
H.~Kamei.
\newblock The existence and classification of synchrony-breaking bifurcations
  in regular homogeneous networks using lattice structures.
\newblock {\em Internat. J. Bifur. Chaos Appl. Sci. Engrg.} {\bf 19(11)} (2009), 3707--3732.

\bibitem{M14}
C.~Moreira.
\newblock On bifurcations in lifts of regular uniform coupled cell networks.
\newblock {\em Proc. R. Soc. Lond. Ser. A Math. Phys. Eng. Sci.} {\bf 470(2169)} (2014).

\bibitem{NRS14}
E.~Nijholt, B.~Rink, and J.~Sanders.
\newblock Graph fibrations and symmetries of network dynamics.
\newblock {\em arXiv preprint arXiv:1410.6021}, 2014.

\bibitem{RS14}
B.~Rink and J.~Sanders.
\newblock Coupled cell networks and their hidden symmetries.
\newblock {\em SIAM J. Math. Anal.} {\bf 46(2)} (2014), 1577--1609.

\bibitem{RS15}
B.~Rink and J.~Sanders.
\newblock Coupled cell networks: semigroups, {L}ie algebras and normal forms.
\newblock {\em Trans. Amer. Math. Soc.} {\bf 367(5)} (2015), 3509--3548.

\bibitem{SGP03}
I.~Stewart, M.~Golubitsky, and M.~Pivato.
\newblock Symmetry groupoids and patterns of synchrony in coupled cell
  networks.
\newblock {\em SIAM J. Appl. Dyn. Syst.} {\bf 2(4)} (2003), 609--646.

\end{thebibliography}
\end{document}